\documentclass[12pt]{article}
\usepackage{amssymb}
\usepackage{amsthm}
\usepackage{amsmath}
\usepackage{graphicx}
\usepackage{enumerate}
\usepackage{pict2e}
\usepackage{framed}

\DeclareMathOperator{\Max}{Max}

\newtheorem{theorem}{Theorem}[section]
\newtheorem{definition}[theorem]{Definition}
\newtheorem{lemma}[theorem]{Lemma}
\newtheorem{proposition}[theorem]{Proposition}

\newtheorem{remark}[theorem]{Remark}
\newtheorem{example}[theorem]{Example}
\newtheorem{corollary}[theorem]{Corollary}
\title{The intuitionistic-like logic based on a poset}
\author{Ivan~Chajda and Helmut~L\"anger$^1$}
\date{}

\begin{document}

\footnotetext[1]{Support of the research of the first author by IGA, project P\v rF~2023~010, and of the second author by the Austrian Science Fund (FWF), project I~4579-N, entitled ``The many facets of orthomodularity'', is gratefully acknowledged.}

\maketitle

\begin{abstract}
The aim of the present paper is to show that the concept of intuitionistic logic based on a Heyting algebra can be generalized in such a way that it is formalized by means of a bounded poset. In this case it is not assumed that the poset is relatively pseudocomplemented. The considered logical connectives negation, implication or even conjunction are not operations in this poset but so-called operators since they assign to given entries not necessarily an element of the poset as a result but a subset of mutually incomparable elements with maximal possible truth values. We show that these operators for negation and implication can be characterized by several simple conditions formulated in the language of posets together with the operator of taking the lower cone. Moreover, our implication and conjunction form an adjoint pair. We call these connectives ``unsharp'' or ``inexact'' in accordance with the existing literature. We also introduce the concept of a deductive system of a bounded poset with implication and prove that it induces an equivalence relation satisfying a certain substitution property with respect to implication. Moreover, the restriction of this equivalence on the base set is uniquely determined by its kernel, i.e.\ the class containing the top element.
\end{abstract}



\section{Introduction}

Intuitionistic logic was algebraically formalized by means of relatively pseudocomplemented semilattices, see \cite{B08}, \cite{B13} and \cite H -- \cite{M70}. If such a semilattice is even a lattice, it is called a Heyting algebra, see \cite H and \cite{M70}. It is well-known that every relatively pseudocomplemented lattice is distributive. The concept of relatively pseudocomplemented lattices was extended by the first author to non-distributive lattices in \cite{C03} under the name sectional pseudocomplementation. It was further extended also for posets in \cite{CLP}. Thus a kind of intuitionistic logic based on sectionally pseudocomplemented lattices was realized.

However, there exist logics based on bounded posets, e.g.\ the logic of quantum mechanics, see e.g.\ \cite{CL22b}, \cite{CL23}, \cite{GG} and \cite{PP}. In particular, orthomodular posets on which the logic of quantum mechanics is based are thoroughly studied in \cite{CL22b}, \cite{CL23}, \cite{Fi} and \cite{PP}.

In order to formalize a logic based on a poset, there are two possible ways how to solve the problem that the operations meet and join need not be defined everywhere. One method is to consider these operations as partial only and then the logical connectives formalized by them and by partial operations derived from them are also only partial. The disadvantage of this approach is that in some cases, such a logic cannot answer the question what is a logical conclusion of some reasoning (or sequential derivation). Hence we prefer another approach, namely we consider so-called operators instead of operations which assign to given entries not necessarily an element as the result, but a certain subset of elements with maximal values. We work with such results of logical connectives which are subsets, but their elements are mutually incomparable. It means that one cannot prefer one element of this set with respect to the other elements. Hence such a logic based on a poset gets an answer concerning logical derivation in each case, but the result may be ``inexact'' or ``unsharp'', see e.g.\ \cite{CLa} or \cite{GG}. We suppose that if an exact logical derivation is impossible from the reasons mentioned above, it is better to have an unsharp result than none. In our opinion unsharp reasoning is an alternative to multiple-valued reasoning which is now generally accepted though it was not accepted by all specialists at first.

In their recent paper \cite{CLa} the authors showed that a certain intuitionistic-like logic can be derived based on arbitrary meet-semilattices satisfying the Ascending Chain Condition (ACC) regardless whether there exist (relative) pseudocomplements or not. We can ask if similar logics may be derived also by means of arbitrary posets with $0$ satisfying the ACC. Within meet-semilattices the logical connective conjunction is usually formalized by the meet operation. Then the unsharp implication as introduced in \cite{CLa} forms an adjoint operator to conjunction. In the case of a poset we must find another operator formalizing conjunction. In \cite{GG} so-called unsharp properties of formal logics were used. In \cite{CLa} we considered unsharpness of implication as well as of negation. This means that for given propositions $p$ and $q$ the results of the implication $p\rightarrow q$ and of the negation $\neg p$, respectively, need not be elements of the corresponding meet-semilattice $\mathbf L$, but may be non-empty subsets of it consisting of mutually incomparable maximal elements. When proceeding from meet-semilattices to posets, we will apply this principle again for the connective conjunction in such a way that implication and conjunction will still be connected by a certain kind of adjointness, see e.g.\ \cite{CL22a}. Due to this fact, in such logics we still have the derivation rule Modus Ponens.

\section{Preliminaries}

In the following we identify singletons with their unique element, i.e.\ we will write $x$ instead of $\{x\}$. Moreover, all posets considered in the sequel are assumed to satisfy the Ascending Chain Condition which we will abbreviate by ACC. This implies that every element lies under a maximal one. Of course, every finite poset satisfies the ACC.

In the sequel we will use the following notation: Let $\mathbf P=(P,\leq)$ be a poset, $a,b,c\in P$ and $A,B$ be non-empty subsets of $P$.
\begin{align*}
      \Max A & :=\text{ set of all maximal elements of }A, \\
      L(a,b) & :=\{x\in P\mid x\leq a,b\}, \\
\Lambda(A,B) & :=\bigcup_{x\in A,y\in B}L(x,y), \\	
     A\leq B & \text{ if }x\leq y\text{ for all }x\in A\text{ and all }y\in B, \\
    A\leq_1B & \text{ if for every }x\in A\text{ there exists some }y\in B\text{ with }x\leq y, \\
       A=_1B & \text{ if both }A\leq_1B\text{ and }B\leq_1A.
\end{align*}
The relations $\leq_1$ and $=_1$ are a quasiorder relation on $2^P$ and an equivalence relation on $2^P$, respectively. It is easy to see that $A\leq_1\Max B$ provided $A\subseteq B$ and that $A\leq_1b$ is equivalent to $A\leq b$.

If a poset $\mathbf P$ has a bottom and a top element, we will denote them by $0$ and $1$, respectively, and we will express this fact by writing $\mathbf P=(P,\leq,0,1)$. Such a poset is called {\em bounded}.

The element $c$ is called the {\em relative pseudocomplement} of $a$ with respect to $b$, formally $c=a*b$, if $c$ is the greatest element $x$ of $P$ satisfying $L(a,x)\leq b$. If $\mathbf P=(P,\leq,0)$ then the relative pseudocomplement $a*0$ of $a$ with respect to $0$ is denoted by $a^*$ and called the {\em pseudocomplement} of $a$, see e.g.\ \cite{C12} or \cite{Fr}. Especially, we have $L(a,a^*)=0$.

\section{Negation derived in posets}

Consider a bounded poset $\mathbf P=(P,\leq,0)$ satisfying the ACC. For $a\in P$ we define
\[
a^0:=\Max\{x\in P\mid L(a,x)=0\}.
\]
Clearly. $a^0$ need not be an element of $P$, but it is a non-empty subset of $P$ since $L(a,0)=0$. From now on, we will call $a^0$ the {\em unsharp negation} of $a$. Of course, if the pseudocomplement $a^*$ of $a$ exists then $a^0=a^*$.

We extend this concept to subsets of $P$ as follows: If $A$ is a non-empty subset of $P$ then
\[
A^0:=\Max\{x\in P\mid L(x,y)=0\text{ for all }y\in A\}.
\]
We are going to show that the negation $^0$ defined in this way shares several properties with the negation in intuitionistic logic.

\begin{theorem}\label{th1}
Let $\mathbf P=(P,\leq,0,1)$ be a bounded poset satisfying the {\rm ACC}, $a,b\in P$ and $A,B$ non-empty subsets of $P$. Then the following holds:
\begin{enumerate}[{\rm(i)}]
\item $A^0$ is an antichain, in particular $a^0$ is an antichain,
\item $0^0=1$ and $1^0=0$,
\item $L(x,y)=0$ for all $x\in A$ and all $y\in A^0$, in particular $L(a,y)=0$ for all $y\in a^0$,
\item $A\leq_1B$ implies $B^0\leq_1A^0$, in particular $a\leq b$ implies $b^0\leq a^0$,
\item $A^0=_1B^0$ implies $A^0=B^0$, in particular $a^0=_1b^0$ implies $a^0=b^0$; $A^0=_1b$ implies $A^0=b$, in particular $a^0=_1b$ implies $a^0=b$,
\item $A\leq_1A^{00}$, in particular $a\leq_1a^{00}$; $A^{000}=A^0$, in particular $a^{000}=a^0$,
\item $\Lambda\Big(a,\big(L(a,b)\big)^0\Big)=_1\Lambda(a,b^0)$.
\end{enumerate}
\end{theorem}

\begin{proof}
\
\begin{enumerate}
\item[(i)] -- (iii) follow directly from the definition of $^0$.
\item[(iv)] If $A\leq_1B$ then
\[
\{x\in P\mid L(x,y)=0\text{ for all }y\in B\}\subseteq\{x\in P\mid L(x,y)=0\text{ for all }y\in A\}.
\]
\item[(v)] Assume $a\in A^0=_1B^0$. Because of $A^0\leq_1B^0$ there exists some $b\in B^0$ with $a\leq b$, and because of $B^0\leq_1A^0$ there exists some $c\in A^0$ with $b\leq c$. Together we obtain $a\leq b\leq c$ and hence $a\leq c$. Since $a$ and $c$ belong to the antichain $A^0$ we conclude $a=c$ and therefore $a=b\in B^0$. This shows $A^0\subseteq B^0$. Interchanging the roles of $A^0$ and $b^0$ yields $B^0\subseteq A^0$. Together we obtain $A^0=B^0$. Now assume $a\in A^0=_1b$. Because of $A^0\leq_1b$ we have $a\leq b$, and because of $b\leq_1A^0$ there exists some $c\in A^0$ with $b\leq c$. Together we obtain $a\leq b\leq c$ and hence $a\leq c$. Since $a$ and $c$ belong to the antichain $A^0$ we conclude $a=c$ and therefore $a=b$. This shows $A^0=b$.
\item[(vi)] We have $A\subseteq\{x\in P\mid L(x,y)=0\text{ for all }y\in A^0\}$ and hence $A\leq_1A^{00}$. Replacing $A$ by $A^0$ we get $A^0\leq_1A^{000}$. From $A\leq_1A^{00}$ and (iv) we obtain $a^{000}\leq_1A^0$. Together we have $A^{000}=_1A^0$ whence $A^{000}=A^0$ according to (v).
\item[(vii)] Any of the following statements implies the next one:
\begin{enumerate}[(1)]
\item $L(x,y)=0$ for all $x\in\big(L(a,b)\big)^0$ and all $y\in L(a,b)$,
\item $L(b,y)=0$ for all $x\in\big(L(a,b)\big)^0$ and all $y\in L(a,x)$,
\item $\Lambda\Big(a,\big(L(a,b)\big)^0\Big)\leq_1b^0$,
\item $\Lambda\Big(a,\big(L(a,b)\big)^0\Big)\leq_1\Lambda(a,b^0)$.
\end{enumerate}
That (1) implies (2) can be seen as follows: If $x\in\big(L(a,b)\big)^0$, $y\in L(a,x)$ and $c\in L(b,y)$ then $c\in L(a,b)$ and $c\leq x$ and therefore $L(x,c)=0$ by (1) whence $L(c)=L(x,c)=0$, i.e. $c=0$. From $L(a,b)\leq b$ we conclude $b^0\leq_1\big(L(a,b)\big)^0$ according to (iv) and hence
\[
\Lambda(a,b^0)\leq_1\Lambda\Big(a,\big(L(a,b)\big)^0\Big).
\]
\end{enumerate}
\end{proof}

\begin{example}
Consider the bounded poset $\mathbf P=(P,\leq,0,1)$ visualized in Fig.~1:

\vspace*{-3mm}

\begin{center}
\setlength{\unitlength}{7mm}
\begin{picture}(6,8)
\put(3,1){\circle*{.3}}
\put(1,3){\circle*{.3}}
\put(3,3){\circle*{.3}}
\put(5,3){\circle*{.3}}
\put(1,5){\circle*{.3}}
\put(3,5){\circle*{.3}}
\put(5,5){\circle*{.3}}
\put(3,7){\circle*{.3}}
\put(3,1){\line(-1,1)2}
\put(3,1){\line(0,1)6}
\put(3,1){\line(1,1)2}
\put(1,3){\line(0,1)2}
\put(1,3){\line(1,1)2}
\put(3,3){\line(-1,1)2}
\put(3,3){\line(1,1)2}
\put(5,3){\line(-1,1)2}
\put(5,3){\line(0,1)2}
\put(3,7){\line(-1,-1)2}
\put(3,7){\line(1,-1)2}
\put(2.85,.3){$0$}
\put(.35,2.85){$a$}
\put(3.4,2.85){$b$}
\put(5.4,2.85){$c$}
\put(.35,4.85){$d$}
\put(3.4,4.85){$e$}
\put(5.4,4.85){$f$}
\put(2.85,7.4){$1$}
\put(2.2,-.75){{\rm Fig.~1}}
\put(1.1,-1.75){Bounded poset}
\end{picture}
\end{center}

\vspace*{8mm}

We have
\[
\begin{array}{l|l|l|l|l|l|l|l|l}
x      & 0 & a & b  & c & d & e & f & 1 \\
\hline
x^0    & 1 & f & ac & d & c & 0 & a & 0 \\
\hline
x^{00} & 0 & a & b  & c & d & 1 & f & 1
\end{array}
\]
{\rm(}Here and in the following we write $a_1\ldots a_n$ instead of $\{a_1,\ldots,a_n\}$.{\rm)} One can see that $x^{00}=x$ for all $x\in P\setminus\{e\}$ and $e^{00}=1\neq e$. But $e^{000}=1^0=0=e^0$ and hence $\mathbf P$ satisfies the identity $x^{000}\approx x^0$ in accordance with {\rm(vi)} of Theorem~\ref{th1}.
\end{example}

The following example shows that (iv) of Theorem~\ref{th1} does not hold for $\leq$ instead of $\leq_1$.

\begin{example}
Consider the bounded poset $\mathbf P=(P,\leq,0,1)$ depicted in Fig.~2:

\vspace*{-3mm}

\begin{center}
\setlength{\unitlength}{7mm}
\begin{picture}(10,8)
\put(5,1){\circle*{.3}}
\put(1,3){\circle*{.3}}
\put(3,3){\circle*{.3}}
\put(5,3){\circle*{.3}}
\put(7,3){\circle*{.3}}
\put(9,3){\circle*{.3}}
\put(5,5){\circle*{.3}}
\put(5,7){\circle*{.3}}
\put(5,1){\line(-2,1)4}
\put(5,1){\line(-1,1)2}
\put(5,1){\line(0,1)6}
\put(5,1){\line(1,1)2}
\put(5,1){\line(2,1)4}
\put(5,5){\line(-1,-1)2}
\put(5,5){\line(1,-1)2}
\put(5,7){\line(-1,-1)4}
\put(5,7){\line(1,-1)4}
\put(4.85,.3){$0$}
\put(.35,2.85){$a$}
\put(2.35,2.85){$b$}
\put(5.4,2.85){$c$}
\put(7.4,2.85){$d$}
\put(9.4,2.85){$e$}
\put(5.4,4.85){$f$}
\put(4.85,7.4){$1$}
\put(4.2,-.75){{\rm Fig.~2}}
\put(3.1,-1.75){Bounded poset}
\end{picture}
\end{center}

\vspace*{8mm}

We have
\[
\begin{array}{l|l|l|l|l|l|l|l|l}
x      & 0 & a  & b    & c    & d    & e  & f  & 1 \\
\hline
x^0    & 1 & ef & acde & abde & abce & af & ae & 0 \\
\hline
x^{00} & 0 & a  & b    & c    & d    & e  & f  & 1
\end{array}
\]
and hence $c\leq f$ and $f^0\leq_1c^0$, but $f^0=\{a,e\}\not\leq\{a,b,d,e\}=c^0$.
\end{example}

The next question is whether the unary operator $^0$ can be characterized by means of simple conditions formulated in the language of posets. This is possible, see the following theorem.

\begin{theorem}\label{th2}
Let $(P,\leq,0,1)$ be a bounded poset satisfying the {\rm ACC} and $^0$ a unary operator on $P$. Then the following conditions {\rm(1)} and {\rm(2)} are equivalent:
\begin{enumerate}[{\rm(1)}]
\item $x^0=\Max\{y\in P\mid L(x,y)=0\}$ for all $x\in P$,
\item the operator $^0\colon2^P\rightarrow2^P$ {\rm(}restricted to $P${\rm)} satisfies the following conditions for all $x\in P$:
\begin{enumerate}[{\rm(P1)}]
\item $x^0$ is an antichain,
\item $L(x,y)=0$ for all $y\in x^0$,
\item $\Lambda\Big(x,\big(L(x,y)\big)^0\Big)=_1\Lambda(x,y^0)$.
\end{enumerate}
\end{enumerate}
\end{theorem}

\begin{proof}
$\text{}$ \\
(1) $\Rightarrow$ (2): \\
This follows from Theorem~\ref{th1}. \\
(2) $\Rightarrow$ (1): \\
If $L(x,y)=0$ then according to (P3) we have
\[
y=_1L(y,1)=\Lambda\Big(y,\big(L(x,y)\big)^0\Big)=\Lambda\Big(y,\big(L(y,x)\big)^0\Big)=_1\Lambda(y,x^0)\leq_1x^0
\]
and hence $y\leq_1x^0$. Conversely, if $y\leq_1x^0$ then according to (P2) we get
\[
L(x,y)\subseteq\Lambda(x,x^0)=0
\]
which implies $L(x,y)=0$. This shows that $L(x,y)=0$ is equivalent to $y\leq_1x^0$. We conclude
\[
\Max\{y\in P\mid L(x,y)=0\}=\Max\{y\in P\mid y\leq_1x^0\}=x^0.
\]
The last equality can be seen as follows. Let $z\in\Max\{y\in P\mid y\leq_1x^0\}$. Then $z\leq_1x^0$, i.e.\ there exists some $u\in x^0$ with $z\leq u$. We have $u\leq_1x^0$. Then $z<u$ would imply $z\notin\Max\{y\in P\mid y\leq_1x^0\}$, a contradiction. This shows $z=u\in x^0$. Conversely, assume $z\in x^0$. Then $z\leq_1x^0$. If $z\notin\Max\{y\in P\mid y\leq_1x^0\}$ then there would exist some $u\in P$ with $z<u\leq_1x^0$ and hence there would exist some $w\in x^0$ with $z<u\leq w$ contradicting (P1). This shows $z\in\Max\{y\in P\mid y\leq_1x^0\}$.
\end{proof}

\section{Unsharp implication and conjunction}

Let us recall that Brouwerion semilattices are relatively pseudocomplemented meet-semilattices, see e.g.\ \cite K and \cite{M55}. It is familiarly known that in logics based on Brouwerian semilattices or on Heyting algebras the relative pseudocomplement is considered as the logical connective implication, i.e.
\[
x\rightarrow y=x*y=\Max\{z\in S\mid x\wedge z\leq y\}.
\]
Of course, $\Max\{z\in S\mid x\wedge z\leq y\}$ is a singleton, thus $x*y$ is an element of $S$. In our case we will use formally the same definition, but the result $x\rightarrow y$ need not be an element of the poset $\mathbf P$ in question, but may be a subset of $P$ in general. However, the elements of $x\rightarrow y$ are mutually incomparable and of a maximal possible value. Hence one cannot pick up one of them to be the preferable element. Now we are going to define our main concept, i.e.\ the operator $\rightarrow$ which formalizes the logical connective implication. As mentioned above, for given entries $x$ and $y$ the result of $x\rightarrow y$ may be a subset of $P$. Due to this, if we combine this operator in several formulas, we must define its value also for entries which are subsets. We define:

Let $\mathbf P=(P,\leq,0,1)$ be a bounded poset satisfying the ACC, $a,b\in P$ and $A,B$ non-empty subsets of $P$. Then
\begin{enumerate}
\item[(I)]
\begin{tabular}l
$a\rightarrow b:=\Max\{x\in P\mid L(a,x)\leq b\}$, \\
$A\rightarrow B:=\Max\{y\in P\mid L(x,y)\leq_1B\text{ for all }x\in A\}$.
\end{tabular}
\end{enumerate}
\begin{enumerate}
\item[(C)]
\begin{tabular}l
$a\odot b:=\Max L(a,b)$, \\
$A\odot B:=\Max \Lambda(A,B)$.
\end{tabular}
\end{enumerate}
Of course, if $\mathbf P$ is a lattice then $a\odot b=a\wedge b$.

It is evident that $a\rightarrow0=a^0$ as usual in intuitionistic logic.

In order to estimate how reasonable our definition of implication is we can verify its relationship with conjunction. These two unsharp logical connectives are related as follows:
\begin{enumerate}
\item[(AD)] $x\odot y\leq z$ if and only if $x\leq_1y\rightarrow z$
\end{enumerate}
or, even more general,
\[
A\odot y\leq z\text{ if and only if }A\leq_1y\rightarrow z
\]
for every non-empty subset $A$ of $P$. This is a variant of adjointness of the operators $\odot$ and $\rightarrow$. Hence the connectives introduced before seem to be sound. We list some of their properties. It is evident that $\odot$ is commutative and $x\odot1\approx1\odot x\approx x$.

Moreover, from $x\rightarrow y=x\rightarrow y$ we infer by (AD)
\[
(x\rightarrow y)\odot x\leq_1y
\]
which is a kind of the Modus Ponens derivation rule in intuitionistic logic. Namely, it says that the truth value of the proposition $y$ cannot be less than the truth values of the propositions $x$ and $x\rightarrow y$ despite the fact that $x\rightarrow y$ may consist of a number of propositions.

For the operator $\rightarrow$ we prove the following result.

\begin{theorem}\label{th3}
Let $\mathbf P=(P,\leq,0,1)$ be a bounded poset satisfying the {\rm ACC}, $a,b,c\in P$ and $A,B$ non-empty subsets of $P$. Then the following holds:
\begin{enumerate}[{\rm(i)}]
\item $a\rightarrow b$ is an antichain,
\item $b\leq_1a\rightarrow b$,
\item $a\leq_1(a\rightarrow b)\rightarrow b$,
\item $a\leq b$ implies $c\rightarrow a\leq_1c\rightarrow b$ and $b\rightarrow c\leq_1a\rightarrow c$,
\item $\Lambda(a,a\rightarrow b)=L(a,b)$,
\item $\Lambda(a\rightarrow b,b)=_1b$,
\item $1\rightarrow A=A$, especially $1\rightarrow a=a$
\item $A\rightarrow B=1$ if and only if $A\leq_1B$, especially $a\rightarrow b=1$ if and only if $a\leq b$,
\item $A\odot b\leq c$ if and only if $A\leq_1b\rightarrow c$,
\item $a\odot(a\rightarrow b)=a\odot b$.
\end{enumerate}
\end{theorem}

\begin{proof}
\
\begin{enumerate}
\item[(i)], (ii), (vii) and (ix) follow directly from the definition of $\rightarrow$.
\item[(iii)] follows from $L(x,a)=L(a,x)\leq b$ for all $x\in a\rightarrow b$
\item[(iv)] follows since $a\leq b$ implies
\begin{align*}
\{x\in P\mid L(c,x)\leq a\}\subseteq\{x\in P\mid L(c,x)\leq b\}, \\
\{x\in P\mid L(b,x)\leq c\}\subseteq\{x\in P\mid L(a,x)\leq c\}.
\end{align*}
\item[(v)] If $c\in\Lambda(a,a\rightarrow b)$ then there exists some $d\in a\rightarrow b$ with $c\in L(a,d)$. Since $L(a,d)\leq b$ we have $c\in L(a,b)$. Conversely, assume $c\in L(a,b)$. Since $L(a,b)\leq b$ there exists some $d\in a\rightarrow b$ with $b\leq d$. Now $c\in L(a,d)\subseteq\Lambda(a,a\rightarrow b)$.
\item[(vi)] Since $L(a,b)\leq b$ there exists some $c\in a\rightarrow b$ with $b\leq c$. Now $b\in L(c,b)\subseteq\Lambda(a\rightarrow b,b)$ showing $b\leq_1\Lambda(a\rightarrow b,b)$. Of course, $\Lambda(a\rightarrow b,b)\leq_1b$.
\item[(viii)] The following are equivalent: $A\rightarrow B=1$, $L(x,1)\leq_1B$ for all $x\in A$, $A\leq_1B$.
\item[(x)] According to (v) we have
\[
a\odot(a\rightarrow b)=\Max\Lambda(a,a\rightarrow b)=\Max L(a,b)=a\odot b.
\]
\end{enumerate}
\end{proof}

We can see that our operator $\rightarrow$ shares a lot of properties with the connective implication in intuitionistic logic based on a Heyting algebra. In particular, our $\rightarrow$ is antitone in the first and monotone in the second entry.

\begin{example}
Consider the bounded poset $(\{0,a,b,c,d,e,1\},\leq,0,1)$ visualized in Fig.~3:

\vspace*{-3mm}

\begin{center}
\setlength{\unitlength}{7mm}
\begin{picture}(7,8)
\put(4,1){\circle*{.3}}
\put(4,3){\circle*{.3}}
\put(6,3){\circle*{.3}}
\put(1,4){\circle*{.3}}
\put(4,5){\circle*{.3}}
\put(6,5){\circle*{.3}}
\put(4,7){\circle*{.3}}
\put(4,1){\line(-1,1)3}
\put(4,1){\line(0,1)6}
\put(4,1){\line(1,1)2}
\put(4,3){\line(1,1)2}
\put(6,3){\line(-1,1)2}
\put(6,3){\line(0,1)2}
\put(4,7){\line(-1,-1)3}
\put(4,7){\line(1,-1)2}
\put(3.85,.3){$0$}
\put(.35,3.85){$c$}
\put(3.35,2.85){$a$}
\put(6.4,2.85){$b$}
\put(3.35,4.85){$d$}
\put(6.4,4.85){$e$}
\put(3.85,7.4){$1$}
\put(3.2,-.75){{\rm Fig.~3}}
\put(2.1,-1.75){Bounded poset}
\end{picture}
\end{center}

\vspace*{8mm}

The operator tables of $\rightarrow$ and $\odot$ are as follows:
\[
\begin{array}{l|l|l|l|l|l|l|l}
\rightarrow & 0  & a  & b  & c  & d  & e  & 1 \\
\hline
0           & 1  & 1  & 1  & 1  & 1  & 1  & 1 \\
\hline
a           & bc & 1  & bc & bc & 1  & 1  & 1 \\
\hline
b           & ac & ac & 1  & ac & 1  & 1  & 1 \\
\hline
c           & de & de & de & 1  & de & de & 1 \\
\hline
d           & c  & ac & bc & c  & 1  & ce & 1 \\
\hline
e           & c  & ac & bc & c  & cd & 1  & 1 \\
\hline
1           & 0  & a  & b  & c  & d  & e  & 1
\end{array}
\hspace{2cm}
\begin{array}{l|l|l|l|l|l|l|l}
\odot & 0 & a & b & c & d  & e  & 1 \\
\hline
0     & 0 & 0 & 0 & 0 & 0  & 0  & 0 \\
\hline
a     & 0 & a & 0 & 0 & a  & a  & a \\
\hline
b     & 0 & 0 & b & 0 & b  & b  & b \\
\hline
c     & 0 & 0 & 0 & c & 0  & 0  & c \\
\hline
d     & 0 & a & b & 0 & d  & ab & d \\
\hline
e     & 0 & a & b & 0 & ab & e  & e \\
\hline
1     & 0 & a & b & c & d  & e  & 1
\end{array}
\]
\end{example}

Now we characterize the binary operator $\rightarrow$ in a similar way as it was done for the unary operator $^0$ in Theorem~\ref{th2}.

\begin{theorem}
Let $\mathbf P=(P,\leq,0,1)$ be a bounded poset satisfying the {\rm ACC} and $\rightarrow$ a binary operator on $P$. Then the following conditions {\rm(1)} and {\rm(2)} are equivalent:
\begin{enumerate}[{\rm(1)}]
\item $x\rightarrow y=\Max\{z\in P\mid L(x,z)\leq y\}$ for all $x,y\in P$,
\item the operator $\rightarrow\colon2^P\times2^P\rightarrow2^P$ {\rm(}restricted to $P\times P${\rm)} satisfies the following conditions for all $x,y,z\in P$:
\begin{enumerate}[{\rm(R1)}]
\item $x\rightarrow y$ is an antichain,
\item $L(x,z)\leq y$ implies $z\leq_1x\rightarrow y$,
\item $\Lambda(x,x\rightarrow y)=L(x,y)$.
\end{enumerate}
\end{enumerate}
\end{theorem}

\begin{proof}
$\text{}$ \\
(1) $\Rightarrow$ (2): \\
This follows from Theorem~\ref{th3}. \\
(2) $\Rightarrow$ (1): \\
According to (R2) we have that $L(x,z)\leq y$ implies $z\leq_1x\rightarrow y$. Conversely, if $z\leq_1x\rightarrow y$ then by (R3) we obtain
\[
L(x,z)\subseteq\Lambda(x,x\rightarrow y)=L(x,y)\leq y
\]
and hence $L(x,z)\leq y$. This shows that $L(x,z)\leq y$ is equivalent to $z\leq_1x\rightarrow y$. We conclude
\[
\Max\{z\in S\mid L(x,z)\leq y\}=\Max\{z\in S\mid z\leq_1x\rightarrow y\}=x\rightarrow y.
\]
The last equality can be seen as follows. Let $u\in\Max\{z\in S\mid z\leq_1x\rightarrow y\}$. Then $u\leq_1x\rightarrow y$, hence there exists some $v\in x\rightarrow y$ with $u\leq v$. We have $v\leq_1x\rightarrow y$. Then $u<v$ would imply $u\notin\Max\{z\in S\mid z\leq_1x\rightarrow y\}$, a contradiction. This shows $u=v\in x\rightarrow y$. Conversely, assume $u\in x\rightarrow y$. Then $u\leq_1x\rightarrow y$. If $u\notin\Max\{z\in S\mid z\leq_1x\rightarrow y\}$ then there would exist some $v\in S$ with $u<v\leq_1x\rightarrow y$ and hence there would exist some $w\in x\rightarrow y$ with $u<v\leq w$ contradicting (R1). This shows $u\in\Max\{z\in S\mid z\leq_1x\rightarrow y\}$.
\end{proof}

Let us note that using both (R2) and (R3) we can prove the adjointness (AD) mentioned above.

\begin{remark}
It is of some interest that the unsharp operator $\rightarrow$ can be characterized by three exact and simple conditions in the language of posets equipped with the operator $L$ of the lower cone.
\end{remark}

\section{Deductive systems}

As mentioned in Section~4, we can involve in our logic the derivation rule Modus Ponens. This rule is in fact closely related to the concept of a deductive system. Hence we define:

\begin{definition}\label{def1}
Let $\mathbf P=(P,\leq,0,1)$ be a bounded poset satisfying the {\rm ACC} and $\rightarrow$ defined by {\rm(I)}.
A {\em deductive system} of $\mathbf P$ is a subset $D$ of $2^P\setminus\{\emptyset\}$ satisfying the following conditions:
\begin{enumerate}[{\rm(i)}]
\item $1\in D$,
\item if $x,y,z,u\in P$, $x\rightarrow y\in D$ and $(x\rightarrow y)\rightarrow(z\rightarrow u)\in D$ then $z\rightarrow u\in D$,
\item if $x,y,z\in P$ and $x\rightarrow y,y\rightarrow x\in D$ then $(z\rightarrow x)\rightarrow(z\rightarrow y)\in D$ and $(x\rightarrow z)\rightarrow(y\rightarrow z)\in D$.
\end{enumerate}
\end{definition}

In the following we make use of the identities $x\odot1\approx1\odot x\approx x$ and $1\rightarrow x\approx x$ (see Section~4).

\begin{remark}
If $x\in D$, $y\in P$ and $x\rightarrow y\in D$ then according to {\rm(vii)} of Theorem~\ref{th3} we have $1\rightarrow x=x\in D$ and $(1\rightarrow x)\rightarrow(1\rightarrow y)=x\rightarrow y\in D$ and hence because of {\rm(ii)} of Definition~\ref{def1} we get $y=1\rightarrow y\in D$. Therefore, $x\in D$ and $x\rightarrow y\in D$ imply $y\in D$ which justifies the name ``deductive system'' and illuminates the connection such systems with the derivation rule Modus Ponens.
\end{remark}

\begin{example}\label{ex1}
Consider the bounded poset $\mathbf P=(P,\leq,0,1)$ depicted in Fig.~4:

\vspace*{-3mm}

\begin{center}
\setlength{\unitlength}{7mm}
\begin{picture}(4,10)
\put(2,1){\circle*{.3}}
\put(1,3){\circle*{.3}}
\put(3,3){\circle*{.3}}
\put(1,5){\circle*{.3}}
\put(1,7){\circle*{.3}}
\put(3,7){\circle*{.3}}
\put(2,9){\circle*{.3}}
\put(2,1){\line(-1,2)1}
\put(2,1){\line(1,2)1}
\put(1,3){\line(0,1)4}
\put(1,3){\line(1,2)2}
\put(3,3){\line(-1,2)2}
\put(3,3){\line(0,1)4}
\put(2,9){\line(-1,-2)1}
\put(2,9){\line(1,-2)1}
\put(1.85,.3){$0$}
\put(.35,2.85){$a$}
\put(3.4,2.85){$b$}
\put(.35,4.85){$c$}
\put(.35,6.85){$d$}
\put(3.4,6.85){$e$}
\put(1.85,9.4){$1$}
\put(1.2,-.75){{\rm Fig.~4}}
\put(.1,-1.75){Bounded poset}
\end{picture}
\end{center}

\vspace*{8mm}

The operator table for $\rightarrow$ is as follows:
\[
\begin{array}{r|r|r|r|r|r|r|r}
\rightarrow & 0 & a & b & c & d & e & 1 \\
\hline
          0 & 1 & 1 & 1 & 1 & 1 & 1 & 1 \\
\hline
          a & b & 1 & b & 1 & 1 & 1 & 1\\
\hline
          b & c & c & 1 & c & 1 & 1 & 1 \\
\hline
          c & b & e & b & 1 & 1 & e & 1 \\
\hline
          d & 0 & a & b & c & 1 & e & 1 \\
\hline
          e & 0 & c & b & c & d & 1 & 1 \\
\hline
          1 & 0 & a & b & c & d & e & 1
\end{array}
\]
We want to show that $D:=\{d,e,1\}$ is a deductive system of $\mathbf P$. {\rm(}The following considerations make heavy use of the operation table for $\rightarrow$.{\rm)} In order to make the proof of this statement short and clear we define two subsets $A$ and $B$ of $P$ as follows:
\begin{align*}
A & :=\{(c,a),(c,e)\}\cup D^2\cup\{(x,y)\in P^2\mid x\leq y\}, \\
B & :=\{a,c\}^2\cup D^2\cup\{(x,x)\mid x\in P\}.
\end{align*}
Now let $x,y,z\in P$. Then we have
\begin{align*}
& x\rightarrow y\in D\text{ if and only if }(x,y)\in A, \\
& x\rightarrow y,y\rightarrow x\in D\text{ if and only if }(x,y)\in B, \\
& 1\in D\text{ according to the definition of }D, \\
& \text{if }x\in D\text{ and }(x,y)\in A\text{ then }y\in D, \\
& \text{if }(x,y)\in B\text{ then }(z\rightarrow x,z\rightarrow y),(x\rightarrow z,y\rightarrow z)\in B
\end{align*}
completing the proof that $D$ is a deductive system of $\mathbf P$.
\end{example}

Since in posets we do not have everywhere defined lattice operations join and meet, we formulate a certain kind of compatibility in the following way.

\begin{definition}
Let $(P,\leq,0,1)$ be a bounded poset satisfying the {\rm ACC} and $\Phi$ be an equivalence relation on $2^P\setminus\{\emptyset\}$. We say that $\Phi$ satisfies the {\em substitution property with respect to $\odot$} if for all $x,y\in P$ we have
\[
1\mathrel{\Phi}x\rightarrow y\hspace*{2mm}\text{ implies }\hspace*{2mm}x\odot1\mathrel{\Phi}x\odot(x\rightarrow y).
\]
We say that $\Phi$ satisfies the {\em substitution property with respect to $\rightarrow$} if for all $x,y,z,u\in P$ we have
\begin{align*}
            x\mathrel{\Phi} y & \text{ implies }x\rightarrow x\mathrel{\Phi}x\rightarrow y, \\
            x\mathrel{\Phi} y & \text{ implies }(z\rightarrow x)\rightarrow(z\rightarrow x)\mathrel{\Phi}(z\rightarrow x)\rightarrow(z\rightarrow y), \\
            x\mathrel{\Phi} y & \text{ implies }(x\rightarrow z)\rightarrow(x\rightarrow z)\mathrel{\Phi}(x\rightarrow z)\rightarrow(y\rightarrow z), \\
1\mathrel{\Phi}x\rightarrow y & \text{ implies }1\rightarrow(z\rightarrow u)\mathrel{\Phi}(x\rightarrow y)\rightarrow(z\rightarrow u).
\end{align*}
\end{definition}

If an equivalence relation on $2^P\setminus\{\emptyset\}$ satisfies the substitution property with respect to $\odot$ and $\rightarrow$ then we can easily describe the relationship to its kernel.

\begin{lemma}\label{lem1}
Let $(P,\leq,0,1)$ be a bounded poset satisfying the {\rm ACC}, $\Phi$ an equivalence relation on $2^P\setminus\{\emptyset\}$ satisfying the substitution property with respect to $\odot$ and $\rightarrow$ and $a,b\in P$. Then $(a,b)\in\Phi$ if and only if $a\rightarrow b,b\rightarrow a\in[1]\Phi$.
\end{lemma}

\begin{proof}
If $(a,b)\in\Phi$ then
\begin{align*}
a\rightarrow b\in[a\rightarrow a]\Phi=[1]\Phi, \\
b\rightarrow a\in[b\rightarrow b]\Phi=[1]\Phi.
\end{align*}
If, conversely, $a\rightarrow b,b\rightarrow a\in[1]\Phi$ then according to (x) of Theorem~\ref{th3} we have
\[
a=a\odot1\mathrel{\Phi}a\odot(a\rightarrow b)=a\odot b=b\odot a=b\odot(b\rightarrow a)\mathrel{\Phi}b\odot1=b.
\]
\end{proof}

When studying congruences in varieties of algebras, an important congruence property is the so-called weak regularity. It means that if an algebra in question has a constant $1$ and if for two of its congruences $\Phi$ and $\Psi$ we have $[1]\Phi=[1]\Psi$ then $\Phi=\Psi$, see e.g.\ \cite{CEL}. Surprisingly, we obtain a similar result for posets and equivalence relations having the substitution property with respect to $\odot$ and $\rightarrow$. In fact, this is a consequence of Lemma~\ref{lem1}.

\begin{corollary}
If $(P,\leq,0,1)$ is a bounded poset satisfying the {\rm ACC}, $\Phi,\Psi$ are equivalence relations on $2^P\setminus\{\emptyset\}$ satisfying the substitution property with respect to $\odot$ and $\rightarrow$ and $[1]\Phi=[1]\Psi$ then $\Phi\cap P^2=\Psi\cap P^2$.
\end{corollary}

For varieties of algebras, the mentioned weak regularity was characterized by B.~Cs\'ak\'any, see e.g.\ \cite{CEL} by means of certain binary terms satisfying a simple condition. It is of some interest that such a binary term can be derived also for posets.

\begin{proposition}
Let $(P,\leq,0,1)$ be a bounded poset satisfying the {\rm ACC}, put $t(x,y):=(x\rightarrow y)\odot(y\rightarrow x)$ for all $x,y\in P$ and let $a,b\in P$. Then $t(a,b)=1$ if and only if $a=b$.
\end{proposition}

\begin{proof}
According to (viii) of Theorem~\ref{th3} the following are equivalent:
\begin{align*}
                                                                                        t(a,b) & =1, \\
                                                         (a\rightarrow b)\odot(b\rightarrow a) & =1, \\
                                                    \Max\Lambda(a\rightarrow b,b\rightarrow a) & =1, \\
                                                                                             1 & \in\Lambda(a\rightarrow b,b\rightarrow a), \\
\text{there exists some }x\in a\rightarrow b\text{ and some }y\in b\rightarrow a\text{ with }1 & \in L(x,y), \\
\text{there exists some }x\in a\rightarrow b\text{ and some }y\in b\rightarrow a\text{ with }x & =y=1, \\
                                                                                             1 & \in(a\rightarrow b)\cap(b\rightarrow a), \\
                                                                                a\rightarrow b & =b\rightarrow a=1, \\
                                                                                             a & \leq b\leq a, \\
                                                                                             a & =b.
\end{align*}
\end{proof}

We are going to show that the kernel of an equivalence relation satisfying the substitution property with respect to $\odot$ and $\rightarrow$ is just a deductive system.

\begin{theorem}\label{th5}
Let $\mathbf P=(P,\leq,0,1)$ be a bounded poset satisfying the {\rm ACC} and $\Phi$ an equivalence relation on $2^P\setminus\{\emptyset\}$ satisfying the substitution property with respect to $\odot$ and $\rightarrow$. Then $[1]\Phi$ is a deductive system of $\mathbf P$.
\end{theorem}

\begin{proof}
Let $a,b,c,d\in P$.
\begin{enumerate}[(i)]
\item This is clear.
\item If $a\rightarrow b,(a\rightarrow b)\rightarrow(c\rightarrow d)\in[1]\Phi$ then according to (vii) of Theorem~\ref{th3} we have
\[
c\rightarrow d=1\rightarrow(c\rightarrow d)\in[(a\rightarrow b)\rightarrow(c\rightarrow d)]\Phi=[1]\Phi.
\]
\item If $a\rightarrow b,b\rightarrow a\in[1]\Phi$ then $(a,b)\in\Phi$ according to Lemma~\ref{lem1} and hence
\begin{align*}
(c\rightarrow a)\rightarrow(c\rightarrow b) & \in[(c\rightarrow a)\rightarrow(c\rightarrow a)]\Phi=[1]\Phi, \\
(a\rightarrow c)\rightarrow(b\rightarrow c) & \in[(a\rightarrow c)\rightarrow(a\rightarrow c)]\Phi=[1]\Phi.
\end{align*}
\end{enumerate}
\end{proof}

The question if a given deductive system $D$ on a bounded poset $(P,\leq,0,1)$ satisfying the ACC induces an equivalence relation on $2^P\setminus\{\emptyset\}$ with kernel $D\cap P$ having a property similar to the substitution property with respect to $\rightarrow$ is answered in the next result. At first, we define the relation induced by $D$.

\begin{definition}
For every bounded poset $(P,\leq,0,1)$ satisfying the {\rm ACC} and every subset $E$ of $2^P\setminus\{\emptyset\}$ define a binary relation $\Theta(E)$ on $2^P\setminus\{\emptyset\}$ as follows:
\[
(A,B)\in\Theta(E)\text{ if and only if }A\rightarrow B,B\rightarrow A\in E
\]
{\rm(}$A,B$ non-empty subsets of $P${\rm)}. We call $\Theta(E)$ the {\em relation induced} by $E$.
\end{definition}

In Theorem~\ref{th5} we proved that for every equivalence relation $\Phi$ on $2^P\setminus\{\emptyset\}$ satisfying the substitution property with respect to $\odot$ and $\rightarrow$, its kernel $[1]\Phi$ is a deductive system of $\mathbf P$. The next theorem shows that there holds some converse version of this result if we consider the restriction of the relation induced by a deductive system of $\mathbf P$ to the base set $P$.

\begin{theorem}
Let $\mathbf P=(P,\leq,0,1)$ be a bounded poset satisfying the {\rm ACC} and $D$ a deductive system of $\mathbf P$. Then the following hold:
\begin{enumerate}[{\rm(i)}]
\item $\Theta(D)\cap P^2$ is an equivalence relation on $P$,
\item $[1]\big(\Theta(D)\cap P^2\big)=D\cap P$,
\item if $x,y,z\in P$ and $(x,y)\in\Theta(D)$ then $(z\rightarrow x,z\rightarrow y)\in\Theta(D)$ and $(x\rightarrow z,y\rightarrow z)\in\Theta(D)$.
\end{enumerate}
\end{theorem}

\begin{proof}
Let $a,b,c\in P$.
\begin{enumerate}[(i)]
\item Evidently, $\Theta(D)$ is reflexive and symmetric. If $(a,b),(b,c)\in\Theta(D)$ then
\[
b\rightarrow c,(b\rightarrow c)\rightarrow(a\rightarrow c),c\rightarrow b,(c\rightarrow b)\rightarrow(c\rightarrow a)\in D
\]
and hence $a\rightarrow c,c\rightarrow a\in D$, i.e.\ $(a,c)\in\Theta(D)$. This shows that $\Theta(D)$ is transitive and therefore an equivalence relation on $P$.
\item The following are equivalent: $a\in[1]\big(\Theta(D)\big)$; $a\rightarrow1,1\rightarrow a\in D$; $1,a\in D$; $a\in D$.
\item follows from Definition~\ref{def1}.
\end{enumerate}
\end{proof}

\begin{example}
For the deductive system $D$ from Example~\ref{ex1} we have
\[
\Theta(D)=\{a,c\}^2\cup D^2\cup\{(x,x)\mid x\in P\}.
\]
\end{example}

Authors' addresses:

Ivan Chajda \\
Palack\'y University Olomouc \\
Faculty of Science \\
Department of Algebra and Geometry \\
17.\ listopadu 12 \\
771 46 Olomouc \\
Czech Republic \\
ivan.chajda@upol.cz

Helmut L\"anger \\
TU Wien \\
Faculty of Mathematics and Geoinformation \\
Institute of Discrete Mathematics and Geometry \\
Wiedner Hauptstra\ss e 8-10 \\
1040 Vienna \\
Austria, and \\
Palack\'y University Olomouc \\
Faculty of Science \\
Department of Algebra and Geometry \\
17.\ listopadu 12 \\
771 46 Olomouc \\
Czech Republic \\
helmut.laenger@tuwien.ac.at


\begin{thebibliography}{99}
\bibitem{B08}
L.~E.~J.~Brouwer, De onbetrouwbaarheid der logische principes. Tijdschrift Wijsbegeerte {\bf2} (1908), 152--158.
\bibitem{B13}
L.~E.~J.~Brouwer, Intuitionism and formalism. Bull.\ Amer.\ Math.\ Soc.\ {\bf20} (1913), 81--96.
\bibitem{C03}
I.~Chajda, An extension of relative pseudocomplementation to non-distributive lattices. Acta Sci.\ Math.\ (Szeged) {\bf69} (2003), 491--496.
\bibitem{C12}
I.~Chajda, Pseudocomplemented and Stone posets. Acta Univ.\ Palack.\ Olomuc.\ Fac.\ Rerum Natur.\ Math.\ {\bf51} (2012), 29--34.
\bibitem{CEL}
I.~Chajda, G.~Eigenthaler and H.~L\"anger, Congruence Classes in Universal Algebra. Heldermann, Lemgo 2012. ISBN 3-88538-226-1.
\bibitem{CL22a}
I.~Chajda and H.~L\"anger, Implication in finite posets with pseudocomplemented sections. Soft Computing {\bf26} (2022), 5945--5953.
\bibitem{CL22b}
I.~Chajda and H.~L\"anger, The logic of orthomodular posets of finite height. Log.\ J.\ IGPL {\bf30} (2022), 143--154.
\bibitem{CL23}
I.~Chajda and H.~L\"anger, Operator residuation in orthomodular posets of finite height. Fuzzy Sets and Systems {\bf467} (2023), 108589 (11 pp.).
\bibitem{CLa}
I.~Chajda and H.~L\"anger, Algebraic structures formalizing the logic with unsharp implication and negation. Logic J.\ IGPL (submitted). http://arxiv.org/abs/2301.02205.
\bibitem{CLP}
I.~Chajda, H.~L\"anger and J.~Paseka, Sectionally pseudocomplemented posets. Order {\bf38} (2021), 527--546.
\bibitem{Fi}
P.~D.~Finch, On orthomodular posets. J.\ Austral.\ Math.\ Soc.\ {\bf11} (1970), 57--62.
\bibitem{Fr}
O.~Frink, Pseudo-complements in semi-lattices. Duke Math.\ J.\ {\bf29} (1962), 505--514.
\bibitem{GG}
R.~Giuntini and H.~Greuling, Toward a formal language for unsharp properties. Found.\ Phys.\ {\bf19} (1989), 931--945.
\bibitem H
A.~Heyting, Die formalen Regeln der intuitionistischen Logik.\ I. Sitzungsberichte Akad.\ Berlin 1930, 42--56.
\bibitem K
P.~K\"ohler, Brouwerian semilattices: the lattice of total subalgebras. Banach Center Publ.\ {\bf9} (1982), 47--56.
\bibitem{M55}
A.~Monteiro, Axiomes ind\'ependants pour les alg\`ebres de Brouwer. Rev.\ Un.\ Mat.\ Argentina {\bf17} (1955), 149--160.
\bibitem{M70}
L.~Monteiro, Les alg\`ebres de Heyting et de Lukasiewicz trivalentes. Notre Dame J.\ Formal Logic {\bf11} (1970), 453--466.
\bibitem{PP}
P.~Pt\'ak and S.~Pulmannov\'a, Orthomodular Structures as Quantum Logics. Kluwer, Dordrecht 1991. ISBN 0-7923-1207-4.
\end{thebibliography}
\end{document}